\documentclass{amsart}

\usepackage{amssymb}
\usepackage{fullpage}
\usepackage{hyperref}
\usepackage[mathcal]{euscript}

\usepackage[alphabetic]{amsrefs}

\newtheorem{theorem}{Theorem}
\newtheorem{lemma}[theorem]{Lemma}

\newtheorem{defn}[theorem]{Definition}

\newcommand{\la}{\leftarrow}
\newcommand{\lra}{\longrightarrow}

\newcommand{\llra}[1]{\stackrel{#1}{\longrightarrow}}

\newcommand{\Ideal}[2]{
\xymatrix{
a \ar@(dl,ul)^{\textstyle{#1}} \ar@(ur,ul)[rr]^{\textstyle{#1}} & & 
b \ar@(dr,ur)_{\textstyle{#1}} \ar@(dl,dr)[ll]^{\textstyle{#2}} 
}}

\newcommand{\invlim}{\operatornamewithlimits{\underset{\longleftarrow}{lim}}}

\newcommand{\sm}{\wedge}        

\newcommand{\hequiv}{\;\simeq\;}       

\newcommand{\Mdef}[2]{\newcommand{#1}{\relax \ifmmode #2 \else $#2$\fi}}

\Mdef{\cA}{\mathcal{A}}
\Mdef{\bI}{{\bf I}}
\Mdef{\cI}{\mathcal{I}}
\Mdef{\cS}{\mathcal{S}}
\Mdef{\cW}{\mathcal{W}}
\Mdef{\bR}{\bf{R}}
\Mdef{\RR}{\mathbf{R}}
\Mdef{\Z}{\mathbf{Z}}
\Mdef{\F}{\mathbf{F}}

\input{xypic}
\xyoption{all}

\begin{document}

\title{Jeff Smith's Theory of Ideals}
\author{Robert R.  Bruner}
\email{robert.bruner@wayne.edu}
\author{Daniel C. Isaksen}
\email{isaksen@math.wayne.edu}
\thanks{}

\date{Jul 16  2006}

\begin{abstract}
Jeff Smith recently proposed a theory of ideals for
rings in a triangulated symmetric monoidal category such
as ring spectra or DGAs.  We show that his definition
is equivalent to a `central'  $R$-$R$-bimodule map $I \to R$.
\end{abstract}

\maketitle

\section{Intro}

In a talk at the Mittag-Leffler Instutute on 6 June 2006,
Jeff Smith~\cite{Smith} proposed a theory of ideals for
rings in a triangulated symmetric monoidal category such
as ring spectra or DGAs.  
We show that his definition
is equivalent to a `central'  $R$-$R$-bimodule map $I \to R$.
(Condition (5) in Lemma 2.)
We are unlikely to return to these matters, so are posting this
just to record this equivalence.

We state, but do not examine, the idea that ideals should be in
one-to-one correspondence with ring quotients $R \to Q$.
Jeff Smith's intended application is a more structured description
of Nil terms in K-theory.
We end with a number of questions.

More recently, Mark Hovey~\cite{Hovey} has some interesting results
on these notions in
his ``Smith ideals of structured ring spectra''.
He has clarified a number of the issues we raise here.

We will speak of spectra, but use only arguments which work in any
internally tensored and cotensored
triangulated symmetric monoidal category $\cS$ satisfying (---Hypotheses---).
Let $S$ be the unit of the symmetric monoidal product.

\section{Ideals}

\begin{defn}
\label{def_ideal}
An {\em Ideal} of a ring spectrum $R$ is a category $\bI$
enriched over $\cS$ with 
\begin{itemize}
\item[-] two objects, $a$ and $b$,
\item[-] morphisms $\bI(a,a) = \bI(a,b) = \bI(b,b) = R$ and $\bI(b,a) = I$,
\end{itemize}
in which the compositions $R \sm R \lra R$ are the
product of $R$.  
\end{defn}

We capitalize the word `Ideal' to distinguish it from the notion
of an ideal of a discrete ring.
Following Jeff Smith, we find it convenient to represent this in a diagram:
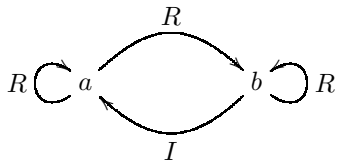
\[
\Ideal{R}{I}
\]

Unpacking the definition, we find that we have 8 compositions, of
which 4 are simply the product map $\mu : R \sm R \lra R$, while the 
other 4 are structure maps:
\[
\begin{tabular}{ p{20ex} lr}
$(a \la a \la b) $ & $  \mu_L : R \sm I \lra I $ \\
$(a \la b \la b) $ & $  \mu_R : I \sm R \lra I $ \\[1ex]
$(a \la b \la a) $ & $  \nu_R : I \sm R \lra R $ \\
$(b \la a \la b) $ & $  \nu_L : R \sm I \lra R $
\end{tabular}
\]

The condition that the objects $a$ and $b$  have identity maps requires that
the diagrams~(\ref{unital}) commute.  The maps labelled $\hequiv$  are the 
unit isomorphisms
of the symmetric monoidal structure, and the 
`inclusions'
$j_L, j_R : I \lra R$
are defined by the diagrams.
\begin{equation}
{\label {unital}}
\xymatrix{
   I \ar[r]^(.3){\hequiv}  \ar@{=}[d]
 & S \sm I \ar[d]^{i \sm 1} 
 & I \ar[l]_(.3){\hequiv} 
     \ar[d]^{j_L} 
 &
 & I \ar[r]^(.3){\hequiv}  \ar@{=}[d]
 & I \sm S \ar[ld]_{\hequiv} \ar[d]^{1 \sm i} 
 & I \ar[l]_<{\hequiv} \ar[d]^{j_R} 
 \\
   I
 & R \sm I \ar[l]^{\mu_L} \ar[r]_{\nu_L} 
 & R 
 &
 & I
 & I \sm R \ar[l]^{\mu_R} \ar[r]_{\nu_R} 
 & R
}
\end{equation}



There are 16 associativity conditions which must be met by the maps $\mu$,
$\mu_L$, $\mu_R$, $\nu_L$, and $\nu_R$ if they are to define an Ideal.
We will label them by the objects ($a$ or $b$) they pass through, as
follows:

{
\begin{itemize}
\addtolength{\itemindent}{20ex}
\item[$aaaa : a \la a \la a \la a$]  
 \qquad  {$\xymatrix{
R \sm R \sm R  \ar[r]^{1 \sm \mu} \ar[d]_{\mu \sm 1} &
R \sm R \ar[d]^{\mu} \\
R \sm R \ar[r]_{\mu} & R \\
} $} \hfill  
\item[$aaab : a \la a \la a \la b$]  
 \qquad  {$\xymatrix{
R \sm R \sm I  \ar[r]^{1 \sm \mu_L} \ar[d]_{\mu \sm 1} &
R \sm I \ar[d]^{\mu_L} \\
R \sm I \ar[r]_{\mu_L} & I \\
} $} \hfill  
\item[$aaba : a \la a \la b \la a$]  
 \qquad  {$\xymatrix{
R \sm I \sm R  \ar[r]^{1 \sm \nu_R} \ar[d]_{\mu_L \sm 1} &
R \sm R \ar[d]^{\mu} \\
I \sm R \ar[r]_{\nu_R} & R \\
} $} \hfill  
\item[$aabb : a \la a \la b \la b$]  
 \qquad  {$\xymatrix{
R \sm I \sm R  \ar[r]^{1 \sm \mu_R} \ar[d]_{\mu_L \sm 1} &
R \sm I \ar[d]^{\mu_L} \\
I \sm R \ar[r]_{\mu_R} & I \\
} $} \hfill  
\item[$abaa : a \la b \la a \la a$]  
 \qquad  {$\xymatrix{
I \sm R \sm R  \ar[r]^{1 \sm \mu} \ar[d]_{\nu_R \sm 1} &
I \sm R \ar[d]^{\nu_R} \\
R \sm R \ar[r]_{\mu} & R \\
} $} \hfill  
\item[$abab : a \la b \la a \la b$]  
 \qquad  {$\xymatrix{
I \sm R \sm I  \ar[r]^{1 \sm \nu_L} \ar[d]_{\nu_R \sm 1} &
I \sm R \ar[d]^{\mu_R} \\
R \sm I \ar[r]_{\mu_L} & I \\
} $} \hfill  
\item[$abba : a \la b \la b \la a$]  
 \qquad  {$\xymatrix{
I \sm R \sm R  \ar[r]^{1 \sm \mu} \ar[d]_{\mu_R \sm 1} &
I \sm R \ar[d]^{\nu_R} \\
I \sm R \ar[r]_{\nu_R} & R \\
} $} \hfill  
\item[$abbb : a \la b \la b \la b$]  
 \qquad  {$\xymatrix{
I \sm R \sm R  \ar[r]^{1 \sm \mu} \ar[d]_{\mu_R \sm 1} &
I \sm R \ar[d]^{\mu_R} \\
I \sm R \ar[r]_{\mu_R} & I \\
} $} \hfill  
\item[$baaa : b \la a \la a \la a$]  
 \qquad  {$\xymatrix{
R \sm R \sm R  \ar[r]^{1 \sm \mu} \ar[d]_{\mu \sm 1} &
R \sm R \ar[d]^{\mu} \\
R \sm R \ar[r]_{\mu} & R \\
} $} \hfill  
\item[$baab : b \la a \la a \la b$]  
 \qquad  {$\xymatrix{
R \sm R \sm I  \ar[r]^{1 \sm \mu_L} \ar[d]_{\mu \sm 1} &
R \sm I \ar[d]^{\nu_L} \\
R \sm I \ar[r]_{\nu_L} & R \\
} $} \hfill  
\item[$baba : b \la a \la b \la a$]  
 \qquad  {$\xymatrix{
R \sm I \sm R  \ar[r]^{1 \sm \nu_R} \ar[d]_{\nu_L \sm 1} &
R \sm R \ar[d]^{\mu} \\
R \sm R \ar[r]_{\mu} & R \\
} $} \hfill  
\item[$babb : b \la a \la b \la b$]  
 \qquad  {$\xymatrix{
R \sm I \sm R  \ar[r]^{1 \sm \mu_R} \ar[d]_{\nu_L \sm 1} &
R \sm I \ar[d]^{\nu_L} \\
R \sm R \ar[r]_{\mu} & R \\
} $} \hfill  
\item[$bbaa : b \la b \la a \la a$]  
 \qquad  {$\xymatrix{
R \sm R \sm R  \ar[r]^{1 \sm \mu} \ar[d]_{\mu \sm 1} &
R \sm R \ar[d]^{\mu} \\
R \sm R \ar[r]_{\mu} & R \\
} $} \hfill  
\item[$bbab : b \la b \la a \la b$]  
 \qquad  {$\xymatrix{
R \sm R \sm I  \ar[r]^{1 \sm \nu_L} \ar[d]_{\mu \sm 1} &
R \sm R \ar[d]^{\mu} \\
R \sm I \ar[r]_{\nu_L} & R \\
} $} \hfill  
\item[$bbba : b \la b \la b \la a$]  
 \qquad  {$\xymatrix{
R \sm R \sm R  \ar[r]^{1 \sm \mu} \ar[d]_{\mu \sm 1} &
R \sm R \ar[d]^{\mu} \\
R \sm R \ar[r]_{\mu} & R \\
} $} \hfill  
\item[$bbbb : b \la b \la b \la b$]  
 \qquad  {$\xymatrix{
R \sm R \sm R  \ar[r]^{1 \sm \mu} \ar[d]_{\mu \sm 1} &
R \sm R \ar[d]^{\mu} \\
R \sm R \ar[r]_{\mu} & R \\
} $} \hfill  
\end{itemize}
}

\begin{lemma}
In any Ideal, the following hold.
\begin{enumerate}
\item $\mu_R$ and $\mu_L$ make $I$ a unital  $R$-$R$-bimodule.
\item $\nu_L$ and $\nu_R$ are $R$-$R$-bimodule homomorphisms.
\item $j_L = j_R$.  We write $j$ for this common value.
\item $j$ is an R-$R$-bimodule map.
\item $\mu_R(1 \sm j) = \mu_L ( j \sm 1)$ (Figure~\ref{jderfig}.) \label{jder}
\item $\nu_L = j \mu_L = \mu (j \sm 1)$ and $\nu_R = j \mu_R = \mu (1 \sm j)$.
\end{enumerate}
Conversely, given an R-$R$-bimodule $I$ and an $R$-$R$-bimodule map $j : I \lra R$
satisfying condition (\ref{jder}), the bimodule structure maps $\mu_L$ and $\mu_R$,
together with the composites
$\nu_L = j \mu_L$ and $\nu_R = j \mu_R$ make an Ideal with morphisms $R$ and $I$ and
compositions $\mu$, $\mu_L$, $\mu_R$, $\nu_L$ and $\nu_R$.
\label{reduction}
\end{lemma}

\begin{proof}
We start by supposing that  {\bf I} is an ideal  of $R$, and
deriving conditions (1) - (6).
Associativities $aaab$ and $abbb$ say that $\mu_L$ and $\mu_R$ are associative, and associativity
$aabb$ says that they commute, making $I$ an $R$-$R$-bimodule.  The identity maps of $a$ and $b$
make it unital on both sides, proving the first item

Associativities $aaba$, $abaa$, $bbab$, and $babb$ show that $\nu_L$ and $\nu_R$ are
$R$-$R$-bimodule homomorphisms.

To show $j_L = j_R$, we write $j_L$ as $\mu(\nu_L \sm 1)(i \sm 1 \sm i)$ preceded by
the equivalence $I \llra{\hequiv} S \sm I \sm S$, as in the following diagram.
\[\xymatrix{
S \sm I \sm S \ar[rd]^{1 \sm 1 \sm i}
              \ar[rdd]_{i \sm 1 \sm i} &
& &
I \sm S \ar[dd]_{j_L \sm 1}
        \ar[dl]_{1 \sm i}
	\ar[lll]_{\hequiv} &
I \ar[l]_(.4){\hequiv}
  \ar[dd]_{j_L}
\\
&
S \sm I \sm R \ar[d]^{i \sm 1 \sm 1} &
I \sm R \ar[l]_(.4){\hequiv}
        \ar[d]_{j_L \sm 1} &
& \\
&
R \sm I \sm R  \ar[r]_{\nu_L \sm 1} 
               \ar[d]_{1 \sm \nu_R } &
R \sm R \ar[d]^{\mu} &
R \sm S \ar[l]_{1 \sm i} &
R \ar[l]_(.4){\hequiv}
  \ar@{=}[lld]
\\
&
R \sm R \ar[r]_{\mu} & 
R & & \\
} \]
Symmetrically, $j_R = \mu(1 \sm \nu_R)(i \sm 1 \sm i)$, and the associativity
$baba$, $\mu(1 \sm \nu_R) = \mu(\nu_L \sm 1)$,
then shows these are equal.  We will now write $j = j_L = j_R.$

To show $j$ is a bimodule homomorphism, it will suffice, by symmetry, to show $j_L$ is
a right $R$-module homomorphism.  This follows since $\nu_L$ is an $R$-$R$-bimodule
homomorphism and
$I \llra{\hequiv} S \sm I \llra{i \sm 1} R \sm I$ is evidently a right
$R$-module homomorphism.

In terms of elements,
the equation $\mu_R(1 \sm j) = \mu_L(j \sm 1)$ says that $j(a)  b = a j(b)$.
Commutativity $abab$ composed with the unit $ S \lra R$
(Figure~\ref{jderfig}) shows that this is true.
\begin{figure}[hb]
\[\xymatrix{
I \sm I \ar@/_1pc/[rrddd]_{j \sm 1}
        \ar[rd]^{\hequiv}
	\ar@/^1pc/[rrrdd]^{1 \sm j} 
& & & \\
 & I \sm S \sm I \ar[rd]^{1 \sm i \sm 1}
 & & \\
 & & I \sm R  \sm I \ar[r]_{1 \sm \nu_L}
                    \ar[d]^{\nu_R \sm 1}
 & I \sm R \ar[d]^{\mu_R} \\
 & & R \sm I \ar[r]_{\mu_L}
 & I \\
}\]
{\caption{\label{jderfig}
The `central', or `derivation' property of $j$ : $j(a)b = aj(b).$}}
\end{figure}

Finally, the diagram below shows that
$\nu_R = j \mu_R$.  Its bottom square is the associativity $baab$ and the rest
is evidently commutative.  The relation $\nu_L = j \mu_L$ is analogous.
\[
\xymatrix{
I \sm R \ar[d]_{\hequiv} \ar[r]^{\mu_R} \ar@/_3pc/[ddd]_{1}
& I \ar[d]_{\hequiv} \ar@/^3pc/[ddd]^{j}
\\
I \sm R \sm S \ar[d]_{1 \sm 1 \sm i}  \ar[r]^{\mu_R \sm 1}
& I \sm S \ar[d]^{1 \sm i}
\\
I \sm R \sm R \ar[d]_{1 \sm \mu} \ar[r]^{\mu_R \sm 1}
& I \sm R \ar[d]_{\nu_R}
\\
I \sm R \ar[r]^{\nu_R}
& R
}
\]
To get $\nu_R = \mu(j \sm 1)$, use the fact that $\nu_R$ is a right 
$R$-module map ($abaa$) and the definition of $j$ : 
\[\xymatrix{
I \sm R \ar[r]^{\hequiv}
        \ar[d]_{j \sm 1}
& I \sm S \sm R \ar[d]^{1 \sm i \sm 1}
\\
R \sm R \ar[d]_{\mu}
& I \sm R \sm R \ar[l]_{\nu_R \sm 1}
                \ar[d]^{1 \sm \mu}
\\
R 
& I \sm R \ar[l]_{\nu_R}
}\]

Now, conversely, suppose given an $R$-$R$-bimodule $I$ (with structure maps $\mu_R$ and
$\mu_L$) and an $R$-$R$-bimodule homomorphism $j : I \lra R$ satisfying
$\mu_R(1 \sm j) = \mu_L(j \sm 1)$.  Define $\nu_R = j \mu_R$
and $\nu_L = j \mu_L$.  We must verify that these are unital and satisfy the 16
associative laws.  

\end{proof}

\section{The main Theorem}

In light of Lemma~\ref{reduction}, we can define an Ideal by specifying
a bimodule map $j : I \lra R$ satisfying the additional property that
$\mu_R(1 \sm j) = \mu_L(j \sm 1)$ ($j(a)b = aj(b)$; see Figure~\ref{jderfig}).

The main theorem is that Ideals of $R$  are in 1-1 correspondence with ring maps
out of $R$.   Jeff Smith outlined an argument for this in his talk, but we do
not reproduce it here.  

\begin{theorem}
The category of Ideals of $R$ is equivalent to the category of rings under
$R$.    The equivalence associates to an Ideal $j$ the cofiber of $j$,
$R \lra R/I = Cof(f)$, and associates to a ring map $f : R \lra Q$ its fiber,
$I = Fib(f) \lra R$.
\end{theorem}

Implicit in the statement of the theorem are several assertions:  that $R/I$ is
a ring and the natural map $R \lra R/I$ a ring map, when $I$ is an Ideal, and
conversely, that the fiber of a ring map is an Ideal.

\section{What else}

The cofiber $R\lra R/I$ should be a ring map in a unique way.

The fiber of a ring hom is exactly an ideal.

If $\mu : R \sm R \lra R$ is an Ideal then it is an equivalence.  Since $\mu$ is a bimodule
homomorphism, this shows that it is insufficient to assume that $I \lra R$ is a bimodule
homomorphism (since $\mu$ need not be an equivalence).  In terms of lemma 2.(4), the map
$\mu$ fails:  $j(a,b) * (c,d) = (abc,d) \neq (a,b) * j(c,d) = (a,bcd).$

Note that for an inclusion $j : I \lra R$ of a sub-R-$R$-bimodule into a discrete ring $R$,
the condition 2.(4) is automatic since $j(x)*y = xy = x*j(y)$.

Standard constructions:\\

1. sum \\

2. product \\

3. prime \\

4. ideal gen by a map $ f : X \lra R$\\
We might want to take the limit of all Ideals $I \lra R$ through which $f$ factors,
but a more direct description as a colimit derived from $f$ would be better.  Small
generators $f$ for common ideals would be good.    EG, 
$I1 + I2 = I(j1+j2 : I1 \wedge Is \lra R$.

Guess: the ideal gen by $\mu : R \sm R \lra R$, $I(\mu) = (R \llra{1_R} R)$.  This
fits with $R/I(\mu) = *$.\\

inverse lim of filtered inverse system is an Ideal if pullbacks are (???)\\
$\invlim I_\alpha \llra{j} R$ is $\invlim I_\alpha \lra I_\alpha \llra{j_\alpha} R$
and this is a bimodule.  Check $j(x)*y = x*j(y)$.\\

5. Connection with EKMM obstruction theory. \\

6. (BAD) $j=2 : S \lra S$ appears to be an Ideal, but $Cof(2)$ is not a
ring spectrum, so something is wrong with the main theorem, identifying
$Cof(j) = End(D)$ when $j : I \lra R$ is an Ideal.\\

7. $j(x)*y - x*j(y) = 0$ looks like we have a derivation which annihilates decomposables.\\
In DGA setting, the cofiber of a discrete bimodule inclusion $j : I \lra R$ is a
square-zero extension, so the $j(xy) = 0$ condition follows from $xy = 0$.  What
about homs of DGAs $R \lra Q$?  Quotient $R/I$ (= colim) is wrong; want hocolim.\\

\end{document}